\documentclass[11pt]{amsart}
\usepackage{mathrsfs}
\usepackage{amssymb,latexsym}
\usepackage{pstcol,pstricks,color}

 \numberwithin{equation}{section}

\newtheorem{theorem}{Theorem}[section]
\newtheorem{proposition}[theorem]{Proposition}
\newtheorem{corollary}[theorem]{Corollary}

\newtheorem{lemma}[theorem]{Lemma}

\newtheorem{example}[theorem]{Example}

\def\qed{\hfill $\Box$}
\def\pf{\noindent {\it Proof.} }

\title{Potential Polynomials and Motzkin Paths}

\begin{document}
\maketitle
\begin{center}
Yidong Sun

Department of Mathematics, Dalian Maritime University, 116026 Dalian, P.R. China\\[5pt]

{\it sydmath@yahoo.com.cn}
\end{center}\vskip0.5cm

\subsection*{Abstract}
A {\em Motzkin path} of length $n$ is a lattice path from $(0,0)$ to
$(n,0)$ in the plane integer lattice $\mathbb{Z}\times\mathbb{Z}$
consisting of horizontal-steps $(1, 0)$, up-steps $(1,1)$, and
down-steps $(1,-1)$, which never passes below the $x$-axis. A {\em
$u$-segment {\rm (resp.}\ $h$-segment {\rm )}} of a Motzkin path is
a maximum sequence of consecutive up-steps ({\rm resp.}\
horizontal-steps). The present paper studies two kinds of statistics
on Motzkin paths: "number of $u$-segments" and "number of
$h$-segments". The Lagrange inversion formula is utilized to
represent the weighted generating function for the number of Motzkin
paths according to the statistics as a sum of the partial Bell
polynomials or the potential polynomials. As an application, a
general framework for studying compositions are also provided.

\medskip

{\bf Keywords}: Partial Bell polynomials, potential polynomials,
Motzkin paths, compositions

\noindent {\sc 2000 Mathematics Subject Classification}: Primary
05A05, 05A15; Secondary 05C90

\section{Introduction}

A {\em Motzkin path} of length $n$ is a lattice path from $(0,0)$ to
$(n,0)$ in the plane integer lattice $\mathbb{Z}\times\mathbb{Z}$
consisting of up-steps $u=(1,1)$, horizontal-steps $h=(1,0)$, and
down-steps $d=(1,-1)$. Denote by $\mathscr{M}_{n}$ the set of
Motzkin paths of length $n$. Let $\mathscr{M}_{n}^{m,k}$ denote the
set of Motzkin paths of length $n$ (i.e. $n=2m+k$) with $m$ up steps
and $k$ horizontal steps and $\mathscr{D}_{m}$ denote the set of
{\em Dyck paths}, namely, Motzkin paths in $\mathscr{M}_{n}^{m,0}$.
Let $P$ be any Motzkin path in $\mathscr{M}_{n}$, a {\em $u$-segment
{\rm (resp.}\ $h$-segment {\rm )}} of $P$ is a maximum sequence of
consecutive up-steps ({\rm resp.}\ horizontal-steps) in $P$ and
define $u_i(P) \ ({\rm resp.}\ h_i(P))$ to be the number of
$u$-segments ({\rm resp.}\ $h$-segments) of length $i$ in $P$ and
call $P$ having the $u$-segments ($h$-segments) of type
$1^{u_1(P)}2^{u_2(P)}\cdots \ ({\rm resp.}\
1^{h_1(P)}2^{h_2(P)}\cdots )$.

In two previous papers\cite{mansun1, mansun2}, we study two kinds of
statistics on ($k$-generalized) Dyck paths: "number of $u$-segments"
and "number of internal $u$-segments". In this paper, we consider
the statistics "number of $u$-segments" and "number of
$h$-segments". In order to do this we present two tools we will use:
the Lagrange inversion formula and the potential polynomials.

\subsection*{Lagrange Inversion Formula~\cite{wilf}}
If $f(x)=\sum_{n\geq 1}f_nx^n$ with $f_1\neq 0$, then the
coefficients of the composition inverse $g(x)$ of $f(x)$ $({\rm
namely}, f(g(x))=g(f(x))=x)$ can be given by
\begin{eqnarray}\label{eqn 1.1}
[x^n]g(x)=\frac{1}{n}[x^{n-1}]\big(\frac{x}{f(x)}\big)^{n}.
\end{eqnarray}
More generally, for any formal power series $\Phi(x)$,
\begin{eqnarray}\label{eqn 1.2}
[x^n]\Phi(g(x))=\frac{1}{n}[x^{n-1}]\Phi'(x)\big(\frac{x}{f(x)}\big)^{n},
\end{eqnarray}
for all $n\geq 1$, where $\Phi'(x)$ is the derivative of $\Phi(x)$
with respect to $x$.
\subsection*{The Potential Polynomials~\cite[pp. 141,157]{comtet}} The potential polynomials
$\textbf{P}_{n}^{(\lambda)}$ related to a given sequence
$(f_n)_{n\geq 1}$ are defined for each complex number $\lambda$ by
\begin{eqnarray*}
1+\sum_{n\geq
1}\textbf{P}_{n}^{(\lambda)}\frac{x^n}{n!}&=&\left\{1+\sum_{n\geq
1}f_n\frac{x^n}{n!}\right\}^\lambda,
\end{eqnarray*}
which can be represented by Bell polynomials
\begin{eqnarray}\label{eqn 1.3}
\textbf{P}_{n}^{(\lambda)}=\textbf{P}_{n}^{(\lambda)}(f_1,f_2,f_3,\dots)=\sum_{1\leq
k\leq n}\binom{\lambda}{k}k!{\bf{B}}_{n,k}(f_1,f_2,f_3,\dots),
\end{eqnarray}
or if $\lambda=r$ is a positive integer, then
\begin{eqnarray}\label{eqn 1.4}
\textbf{P}_{n}^{(r)}=\textbf{P}_{n}^{(r)}(f_1,f_2,f_3,\dots)
=\binom{n+r}{r}^{-1}{\bf{B}}_{n+r,r}(1,2f_1,3f_2,4f_3,\dots),
\end{eqnarray}
where ${\bf B}_{n,i}\big(x_1,x_2,\cdots\big)$ is the partial Bell
polynomial on the variables $\{x_j\}_{j\geq 1}$ \cite{Be}, which has
the explicit formula
\begin{eqnarray}\label{eqn 1.5}
{\bf
B}_{n,r}\big(x_1,x_2,\cdots\big)=\sum_{\sigma_n(r)}\frac{n!}{r_1!r_2!\cdots
r_n!}
\left(\frac{x_1}{1!}\right)^{r_1}\left(\frac{x_2}{2!}\right)^{r_2}\cdots\left(\frac{x_n}{n!}\right)^{r_n},
\end{eqnarray}
where the summation $\sigma_n(r)$ is for all the nonnegative integer
solutions of $r_1+r_2+\cdots+r_n=r$ and $r_1+2r_2+\cdots +nr_n=n$.

In this paper, using the Lagrange inversion formula, we can
represent the generating functions for the number of Motzkin paths
according to our statistics (see Sections 2) as a sum of partial
Bell polynomials or the potential polynomials, for example
\begin{eqnarray*}
\lefteqn{\sum_{P\in\mathscr{M}_{n}^{m,k}}\prod_{i\geq1}t_i^{u_i(P)}\prod_{i\geq1}s_i^{h_i(P)}} \\
&=&\sum_{j=0}^{k}\sum_{\ell=j}^k(-1)^{\ell-j}{\binom{\ell-1}{\ell-j}\binom{m+j}{j}}\frac{{\bf
P}_{m}^{(m+j+1)}\big(1!t_1,2!t_2,\cdots\big)}{(m+1)!}\frac{\ell!{\bf
B}_{k,\ell}\big(1!s_1,2!s_2,\cdots\big)}{k!}.
\end{eqnarray*}
Many important special cases are considered which generate several
surprising results. As an application (see Section 3), compositions
can be regarded as a kind of special Motzkin paths, which leads to a
general framework to studying compositions by specializing the
parameters. Moreover, in the last section we generalize compositions
to matrix compositions.

\vskip0.5cm
\section{"$u$-segments" and "$h$-segments" statistics}

We start this section by studying the ordinary generating function
for the number of Motzkin paths of length $n$ according to the
statistics $u_1,u_2,\ldots$ and $h_1,h_2,\ldots$, that is,
\begin{eqnarray*}
M(x,y;{\bf
t;s})=M(x,y;t_1,t_2,\ldots;s_1,s_2,\ldots)=\sum_{m,k\geq0}x^my^k
\sum_{P\in\mathscr{M}_{n}^{m,k}}\prod_{i\geq1}t_i^{u_i(P)}\prod_{i\geq1}s_i^{h_i(P)}.
\end{eqnarray*}

\begin{proposition}\label{pro2.1}
The generating function $M(x,y;{\bf t;s})$ satisfies the functional
recurrence relation
\begin{eqnarray}\label{eqn 2.1}
M(x,y;{\bf t;s})&=&\frac{T(z)}{1-\frac{S(y)-1}{S(y)}T(z)},
\end{eqnarray}
where $T(x)=1+\sum_{i\geq 1}t_ix^i$, $S(y)=1+\sum_{i\geq 1}s_iy^i$
and $z=xM(x,y;{\bf t;s})$.
\end{proposition}
\begin{proof}
Let $P:=P(x,y;{\bf t;s})$ and $Q:=Q(x,y;{\bf t;s})$ be the ordinary
generating functions for the set of Motzkin paths beginning with
up-steps and with horizontal steps respectively, according to the
statistics $u_1,u_2,\ldots$ and $h_1,h_2,\ldots$. Then $M(x,y;{\bf
t;s})$ satisfies
\begin{eqnarray}\label{eqn 2.2}
M(x,y;{\bf t;s})=1+P+Q.
\end{eqnarray}
Note that $P(x,y;{\bf t;s})$ can be written as
$$P(x,y;{\bf t;s})=\sum_{j\geq1}P_j(x,y;{\bf t;s}),$$ where
$P_j(x,y;{\bf t;s})$ is the ordinary generating function for the
number of Motzkin paths starting with $j$ up-steps according to the
statistics $u_1,u_2,\ldots$ and $h_1,h_2,\ldots$. An equation for
the generating function $P_j(x,z;{\bf t;s})$ is obtained from the
first return decomposition of a Motzkin path $M$ starting with a
$u$-segment of length $j$: either
$$M=u^jdP^{(j)}d\ldots dP^{(2)}dP^{(1)}\mbox{ or }M=
u^jQ^{(j+1)} d P^{(j)} d\ldots dP^{2}dP^{1},$$ where
$P^{(1)},\ldots,P^{(j)}$ are Motzkin paths and $Q^{(j+1)}$ is a
Motzkin path beginning with horizontal steps, see
Figure~\ref{fDD2.1}.
\begin{figure}[h]
\begin{pspicture}(6.2,2)
\psline[unit=18pt,linewidth=.5pt]{->}(0,0)(9.5,0)\psline[unit=18pt,linewidth=.5pt]{->}(0,0)(0,1.8)
\psline[unit=18pt,linewidth=.5pt](0,0)(1.75,1.5)(2,1.25)
\put(1.15,.65){\psline[unit=17pt,linewidth=1pt,fillstyle=solid,fillcolor=gray](1.25,.25)(.25,.25)(.5,.5)(.75,.45)(1,.5)(1.25,.25)}
\psline[unit=18pt,linewidth=.5pt](3,1.25)(3.25,1)\pscircle*[unit=18pt,linewidth=.5pt](3.36,.98){0.03}
\pscircle*[unit=18pt,linewidth=.5pt](3.5,.98){0.03}\pscircle*[unit=18pt,linewidth=.5pt](3.66,.98){0.03}
\psline[unit=18pt,linewidth=.5pt](3.75,1)(4,.75)
\multiput(-1.38,-.18)(.8,-.16){4}{\put(3.78,.5){\psline[unit=18pt,linewidth=1pt,fillstyle=solid,fillcolor=gray](1.25,.25)(.25,.25)(.5,.5)(.75,.45)(1,.5)(1.25,.25)}}
\multiput(0,0)(.8,-.16){3}{\psline[unit=18pt,linewidth=.5pt](5,.75)(5.25,.5)}
\put(1.4,1.05){\tiny$P^{(j)}$} \put(2.7,.73){\tiny$P^{(4)}$}
\setlength\unitlength{10pt}
\rput*[r](4.65,.55){\tiny$P^{(2)}$}\put(15,.7){\tiny$P^{(1)}$}
\end{pspicture}
\begin{pspicture}(6.2,2)
\psline[unit=18pt,linewidth=.5pt]{->}(0,0)(11,0)\psline[unit=18pt,linewidth=.5pt]{->}(0,0)(0,1.8)
\psline[unit=18pt,linewidth=.5pt](0,0)(1.75,1.5)(2.25,1.5)
\put(1.3,.8){\psline[unit=17pt,linewidth=1pt,fillstyle=solid,fillcolor=gray](1.25,.25)(.25,.25)(.5,.5)(.75,.45)(1,.5)(1.25,.25)}
\psline[unit=18pt,linewidth=.5pt](3.25,1.5)(3.5,1.25)\put(1.4,1.22){\tiny$Q^{(j+1)}$}
\put(2.1,.66){\psline[unit=17pt,linewidth=1pt,fillstyle=solid,fillcolor=gray](1.25,.25)(.25,.25)(.5,.5)(.75,.45)(1,.5)(1.25,.25)}
\put(.95,0){\psline[unit=18pt,linewidth=.5pt](3,1.25)(3.25,1)\pscircle*[unit=18pt,linewidth=.5pt](3.36,.98){0.03}
\pscircle*[unit=18pt,linewidth=.5pt](3.5,.98){0.03}\pscircle*[unit=18pt,linewidth=.5pt](3.66,.98){0.03}
\psline[unit=18pt,linewidth=.5pt](3.75,1)(4,.75)
\multiput(-1.38,-.18)(.8,-.16){4}{\put(3.78,.5){\psline[unit=18pt,linewidth=1pt,fillstyle=solid,fillcolor=gray](1.25,.25)(.25,.25)(.5,.5)(.75,.45)(1,.5)(1.25,.25)}}
\multiput(0,0)(.8,-.16){3}{\psline[unit=18pt,linewidth=.5pt](5,.75)(5.25,.5)}
\put(1.5,1.12){\tiny$P^{(j)}$}\put(2.7,.73){\tiny$P^{(4)}$}
\setlength\unitlength{10pt}
\rput*[r](4.65,.55){\tiny$P^{(2)}$}\put(15,.7){\tiny$P^{(1)}$}}
\end{pspicture}

\caption{First return decomposition of a Motzkin path starting with
exactly $j$ up-steps.}\label{fDD2.1}
\end{figure}
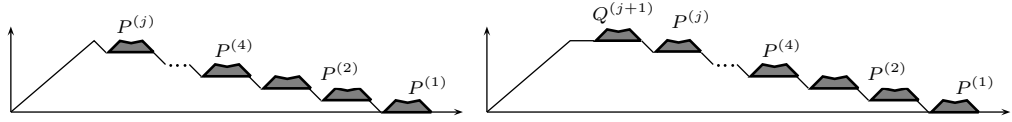

Thus $P_j(x,y;{\bf t})=(1+Q)x^{j}t_{j}M^{j}(x,y;{\bf t;s})$. Hence,
$P(x,y;{\bf t;s})$ satisfies
\begin{eqnarray}\label{eqn 2.3}
P(x,y;{\bf t;s})=(1+Q)\sum_{j\geq1}t_jx^{j}M^{j}(x,y;{\bf t;s}).
\end{eqnarray}
Similarly, one can derive that $Q(x,y;{\bf t;s})$ satisfies
\begin{eqnarray}\label{eqn 2.4}
Q(x,y;{\bf t;s})=(1+P)\sum_{j\geq1}s_jy^{j}.
\end{eqnarray}
Define $T(x)=1+\sum_{i\geq 1}t_ix^i$, $S(y)=1+\sum_{i\geq 1}s_iy^i$
and $z=xM(x,y;{\bf t;s})$. By (\ref{eqn 2.2})-(\ref{eqn 2.4}), one
can deduce (\ref{eqn 2.1}), as required.
\end{proof}

\begin{theorem}\label{theo 2.2}
For any integers $n,m,k\geq 0$, there holds
\begin{eqnarray*}
\lefteqn{\sum_{P\in\mathscr{M}_{n}^{m,k}}\prod_{i\geq1}t_i^{u_i(P)}\prod_{i\geq1}s_i^{h_i(P)}} \\
&=&\sum_{j=0}^{k}\sum_{\ell=j}^k(-1)^{\ell-j}{\binom{\ell-1}{\ell-j}\binom{m+j}{j}}\frac{{\bf
P}_{m}^{(m+j+1)}\big(1!t_1,2!t_2,\cdots\big)}{(m+1)!}\frac{\ell!{\bf
B}_{k,\ell}\big(1!s_1,2!s_2,\cdots\big)}{k!}.
\end{eqnarray*}
\end{theorem}
\begin{proof}
Applying the Lagrange inversion formula (\ref{eqn 1.2}) to (\ref{eqn
2.1}), by the identity
\begin{eqnarray}\label{eqn vandcon}
\sum_{i=0}^j(-1)^i\binom{j}{i}\binom{-i}{\ell}=(-1)^{\ell-j}\binom{\ell-1}{\ell-j},
\end{eqnarray}
we obtain
\begin{eqnarray*}
\lefteqn{\sum_{P\in\mathscr{M}_{n}^{m,k}}\prod_{i\geq1}t_i^{u_i(P)}\prod_{i\geq1}s_i^{h_i(P)}=
[x^{m+1}y^k]xM(x,y;{\bf t;s})=[x^{m+1}y^k]z }\\
&=&\frac{1}{m+1}[x^{m}y^k]\left\{\frac{T(x)}{1-\frac{S(y)-1}{S(y)}T(x)}\right\}^{m+1}\\
&=& \frac{1}{m+1}\sum_{j=0}^k{\binom{m+j}{j}}[x^{m}y^{k}]T(x)^{m+j+1}\left\{\frac{S(y)-1}{S(y)}\right\}^{j}\\
&=&\frac{1}{m+1}\sum_{j=0}^k\sum_{i=0}^{j}(-1)^i{\binom{m+j}{j}\binom{j}{i}}
[x^{m}]T(x)^{m+j+1}[y^{k}]S(y)^{-i}\\
&=&\sum_{j=0}^k\binom{m+j}{j}\frac{{\bf
P}_{m}^{(m+j+1)}\big(1!t_1,2!t_2,\cdots\big)}{(m+1)!}\sum_{i=0}^{j}(-1)^i\binom{j}{i}\frac{{\bf
P}_{k}^{(-i)}\big(1!s_1,2!s_2,\cdots\big)}{k!}\\
&=&\sum_{j=0}^k\binom{m+j}{j}\frac{{\bf
P}_{m}^{(m+j+1)}\big(1!t_1,2!t_2,\cdots\big)}{(m+1)!}\sum_{i=0}^{j}(-1)^i\binom{j}{i}\sum_{\ell=0}^k\binom{-i}{\ell}
\frac{\ell!{\bf B}_{k,\ell}\big(1!s_1,2!s_2,\cdots\big)}{k!}\\
&=&\sum_{j=0}^k\binom{m+j}{j}\frac{{\bf
P}_{m}^{(m+j+1)}\big(1!t_1,2!t_2,\cdots\big)}{(m+1)!}\sum_{\ell=0}^k\frac{\ell!{\bf
B}_{k,\ell}\big(1!s_1,2!s_2,\cdots\big)}{k!}\sum_{i=0}^{j}(-1)^i\binom{j}{i}\binom{-i}{\ell}\\
&=&\sum_{j=0}^k\binom{m+j}{j}\frac{{\bf
P}_{m}^{(m+j+1)}\big(1!t_1,2!t_2,\cdots\big)}{(m+1)!}\sum_{\ell=0}^k\frac{\ell!{\bf
B}_{k,\ell}\big(1!s_1,2!s_2,\cdots\big)}{k!}(-1)^{\ell-j}\binom{\ell-1}{\ell-j}\\
&=&\sum_{j=0}^{k}\sum_{\ell=j}^k(-1)^{\ell-j}{\binom{\ell-1}{\ell-j}\binom{m+j}{j}}\frac{{\bf
P}_{m}^{(m+j+1)}\big(1!t_1,2!t_2,\cdots\big)}{(m+1)!}\frac{\ell!{\bf
B}_{k,\ell}\big(1!s_1,2!s_2,\cdots\big)}{k!},
\end{eqnarray*}
as claimed.
\end{proof}

Let $\mathscr{M}_{n,r,\ell}^{m,k}$ be the subset of
$\mathscr{M}_{n}^{m,k}$ with $r$ number of $u$-segments and $\ell$
number of $h$-segments. Note that ${\bf
B}_{m,r}\big(1!qt_1,2!qt_2,\cdots\big)=q^r{\bf
B}_{m,r}\big(1!t_1,2!t_2,\cdots\big)$ by (\ref{eqn 1.5}), combining
(\ref{eqn 1.3}) and (\ref{eqn vandcon}). Then Theorem \ref{theo 2.2}
produces
\begin{corollary}\label{coro 2.3}
For any integers $n,m,r,k,\ell\geq 0$, there holds
\begin{eqnarray*}
\sum_{P\in\mathscr{M}_{n,r,\ell}^{m,k}}\prod_{i\geq1}t_i^{u_i(P)}\prod_{i\geq1}s_i^{h_i(P)}
=\frac{r!\ell!V_{m,k}^{r,\ell}}{k!(m+1)!}{\bf
B}_{m,r}\big(1!t_1,2!t_2,\cdots\big){\bf
B}_{k,\ell}\big(1!s_1,2!s_2,\cdots\big),
\end{eqnarray*}
where
\begin{eqnarray*}
V_{m,k}^{r,\ell}=\sum_{j=0}^{k}(-1)^{\ell-j}\binom{\ell-1}{\ell-j}
{\binom{m+j}{m}}{\binom{m+j+1}{r}}.
\end{eqnarray*}
\end{corollary}
Recall that
\begin{eqnarray*}
{\bf
B}_{m,r}\big(x_1,x_2,\cdots\big)=\sum_{\sigma_m(r)}\frac{m!}{r_1!r_2!\cdots
r_m!}
\left(\frac{x_1}{1!}\right)^{r_1}\left(\frac{x_2}{2!}\right)^{r_2}\cdots\left(\frac{x_m}{m!}\right)^{r_m},
\end{eqnarray*}
where the summation $\sigma_m(r)$ is for all the nonnegative integer
solutions of $r_1+r_2+\cdots+r_m=r$ and $r_1+2r_2+\cdots +mr_m=m$.

If  comparing the coefficient of $t_1^{r_1}t_2^{r_2}\cdots
t_m^{r_m}s_1^{\ell_1}s_2^{\ell_2}\cdots s_k^{\ell_k}$ in Corollary
\ref{coro 2.3}, one can obtain that
\begin{corollary}
The number of Motzkin paths in $\mathscr{M}_{n,r,\ell}^{m,k}$ with a
number $r$ of $u$-segments of type $1^{r_1}2^{r_2}\cdots m^{r_m}$
and a number $\ell$ of $h$-segments of type
$1^{\ell_1}2^{\ell_2}\cdots k^{\ell_k}$ is
\begin{eqnarray*}
\frac{1}{m+1}\binom{r}{r_1,r_2,\cdots,r_m}\binom{\ell}{\ell_1,\ell_2,\cdots,\ell_k}
V_{m,k}^{r,\ell}.
\end{eqnarray*}
\end{corollary}

Next, specialization for $T(x)$ and $S(y)$ are considered, which
generate several interesting results, as described in
Examples~\ref{ex1}-\ref{ex2}.

\begin{example}\label{ex1} Let $T(x)=e^x, S(y)=e^{y}$, that is, $t_i=s_i={1/i!}$ for all
$i\geq 1$. And Stirling numbers $S(k,i)$ of the second kind satisfy
$(e^x-1)^i/i!=\sum_{k\geq i}S(k,i){x^k/k!}$. Then Theorem \ref{theo
2.2} gives
\begin{eqnarray*}
\sum_{P\in\mathscr{M}_{n}^{m,k}}\prod_{i\geq1}\left\{\frac{1}{i!}\right\}^{u_i(P)+h_i(P)}
=\sum_{j=0}^k(-1)^{k-j}\binom{m+j}{j}\frac{j!(m+j+1)^m}{k!(m+1)!}S(k,j).
\end{eqnarray*}
Note that ${\bf B}_{k,i}\big(1,1,1,\cdots\big)=S(k,i)$
\cite[pp.135]{comtet}, by Corollary \ref{coro 2.3}, we have
\begin{eqnarray*}
\sum_{P\in\mathscr{M}_{n,r,\ell}^{m,k}}\prod_{i\geq1}\left\{\frac{1}{i!}\right\}^{u_i(P)+h_i(P)}
=\frac{r!\ell!V_{m,k}^{r,\ell}}{k!(m+1)!}S(m,r)S(k,\ell).
\end{eqnarray*}
\end{example}

\begin{example}\label{ex2} Let $T(x)=f(x), S(y)=g(y)$, where $f(x), g(y)$ are the
generating function for the complete $b$-ary and $d$-ary plane trees
(see, for instance, \cite{klarner} and \cite[pp. 112-113]{goulden}),
which satisfies the relations $f(x)=1+xf^b(x)$ and $g(y)=1+yg^d(y)$
respectively. By the Lagrange inversion formula (\ref{eqn 1.2}), one
can deduce $t_i=\frac{1}{bi+1}\binom{bi+1}{i}$ and
$s_i=\frac{1}{di+1}\binom{di+1}{i}$. Then Theorem \ref{theo 2.2}
leads to
\begin{eqnarray*}
\lefteqn{\sum_{P\in\mathscr{M}_{n}^{m,k}}\prod_{i\geq1}\left\{\frac{1}{bi+1}\binom{bi+1}{i}\right\}^{u_i(P)}
\prod_{i\geq1}\left\{\frac{1}{di+1}\binom{di+1}{i}\right\}^{h_i(P)} }\\
&=&\frac{1}{m+1}\sum_{j=0}^k{\binom{m+j}{j}}\frac{m+j+1}{(b+1)m+j+1}
\binom{(b+1)m+j+1}{m}\frac{dj-j}{dk-j}\binom{dk-j}{k-j},
\end{eqnarray*}
which, in the case $d=1$, generates
\begin{eqnarray*}
\sum_{P\in\mathscr{M}_{n}^{m,k}}\prod_{i\geq1}\left\{\frac{1}{bi+1}\binom{bi+1}{i}\right\}^{u_i(P)}
=\frac{1}{(b+1)m+k+1}{\binom{m+k+1}{k}}\binom{(b+1)m+k+1}{m}.
\end{eqnarray*}
\end{example}

Recently, Abbas and Bouroubi \cite{abbas} derived two new identities
for Bell polynomials, that is,
\begin{lemma}\label{lemma 1.1}
Let $f(x)=1+\sum_{i\geq 1}f_ix^i$ be any analytic function about the
origin and define $f_m(i)=D^m[f(x)]^i|_{x=0}$, where $D$ is the
differential operator $d/dx$. Then for any integers $m\geq r\geq 1$,
there holds,
\begin{eqnarray}\label{eqn abbas1}
{\bf{B}}_{m,r}(1,f_1(2),f_2(3),\cdots)=\binom{m-1}{r-1}f_{m-r}(m).
\end{eqnarray}
\end{lemma}
\begin{lemma}\label{lemma 1.2}
Let $\{\phi_n(x)\}_{n\geq 0}$ be a binomial sequence. Then for any
integers $m\geq r\geq 1$, there holds,
\begin{eqnarray}\label{eqn abbas2}
{\bf{B}}_{m,r}(1,2\phi_1(1),3\phi_2(1),\cdots)=\binom{m}{r}\phi_{m-r}(r).
\end{eqnarray}
\end{lemma}
Recall that a sequence of polynomials $\{\phi_n(x)\}_{n\geq 0}$ with
$\phi_0(x)=1$ is called binomial if
\begin{eqnarray*}
\phi_n(x+y)=\sum_{i=0}^n\binom{n}{i}\phi_i(x)\phi_{n-i}(y),
\end{eqnarray*}
or equivalently, there exists a power series $\lambda(u)=\sum_{i\geq
1}\lambda_iu^i$ with $\lambda_1\neq 0$ such that
\begin{eqnarray*}
\sum_{n\geq 0}\phi_n(x)\frac{u^n}{n!}=\exp{(x\lambda(u))}.
\end{eqnarray*}
For examples, the following binomial sequences are well known
\cite{stanley},
\begin{itemize}
\item Power polynomials $\phi_n(x)=x^n$;
\item Factorial polynomials $\phi_n(x)=x(x+1)\cdots(x+n-1)$;
\item Abel polynomials $\phi_n(x)=x(x-qn)^{n-1}$ for fixed $q$;
\item Exponential polynomials $\phi_n(x)=\sum_{i=0}^nS(n,i)x^i$.
\end{itemize}

Let $t_i=\frac{f_i(i+1)}{(i+1)!}$ and $s_i=\frac{g_{i-1}(i)}{i!}$,
where $g_m(i)=D^m[g(x)]^i|_{x=0}$ and $g(x)=1+\sum_{i\geq 1}g_ix^i$,
using (\ref{eqn 1.4}) and (\ref{eqn vandcon}), by Theorem \ref{theo
2.2} and Lemma \ref{lemma 1.1}, one can deduce that
\begin{corollary}
For any integers $n,m,k\geq 0$, there holds
\begin{eqnarray*}
\lefteqn{\sum_{P\in\mathscr{M}_{n}^{m,k}}\prod_{i\geq1}\left\{\frac{f_i(i+1)}{(i+1)!}\right\}^{u_i(P)}
\prod_{i\geq1}\left\{\frac{g_{i-1}(i)}{i!}\right\}^{h_i(P)} }\\
&=&\sum_{j=0}^k\sum_{\ell=j}^{k}(-1)^{\ell-j}\binom{\ell}{j}\frac{g_{k-\ell}(k)}{(k-\ell)!}
\frac{j(m+j+1)}{k(2m+j+1)}{\binom{m+j}{j}}\frac{f_m(2m+j+1)}{(m+1)!},
\end{eqnarray*}
which, in the case $g_{m}(i)=\binom{i}{m}m!$ for all $i\geq 1$, by
the identity
\begin{eqnarray}\label{eqn 2.8}
\sum_{\ell=j}^{k}(-1)^{\ell-j}\binom{\ell}{j}\binom{k}{\ell}=\delta_{k,j},
\end{eqnarray}
leads to
\begin{eqnarray}\label{eqn 2.9}
\sum_{P\in\mathscr{M}_{n}^{m,k}}\prod_{i\geq1}\left\{\frac{f_i(i+1)}{(i+1)!}\right\}^{u_i(P)}
&=&\frac{1}{2m+k+1}{\binom{m+k+1}{k}}\frac{f_m(2m+k+1)}{m!}.
\end{eqnarray}
\end{corollary}
\begin{example}\label{ex 2.10}
Let $f_r(x)=\sum_{i\geq 0}(ri+1)^{i-1}\frac{x^i}{i!}$, which is the
exponential generating function for rooted complete $r$-ary labeled
trees for $r\geq 0$ and satisfies the relation
$f_r(x)=e^{xf_r(x)^r}$. By the Lagrange inversion formula, one can
deduce $f_{r,m}(i)=i(rm+i)^{m-1}$. Then (\ref{eqn 2.9}) produces
\begin{eqnarray*}
\sum_{P\in\mathscr{M}_{n}^{m,k}}\prod_{i\geq1}\left\{\frac{((r+1)i+1)^{i-1}}{i!}\right\}^{u_i(P)}
&=&{\binom{m+k+1}{k}}\frac{((r+2)m+k+1)^{m-1}}{m!}.
\end{eqnarray*}
\end{example}

Let $t_i=\frac{\phi_i(1)}{i!}$ and
$s_i=\frac{\psi_{i-1}(1)}{(i-1)!}$, where $\{\phi_n(x)\}_{n\geq 0}$
and $\{\psi_n(x)\}_{n\geq 0}$ are binomial sequences, using
(\ref{eqn 1.4}) and (\ref{eqn vandcon}), by Theorem \ref{theo 2.2}
and Lemma \ref{lemma 1.2}, one can deduce that
\begin{corollary}\label{coro 2.11}
For any integers $n,m,k\geq 0$, there holds
\begin{eqnarray*}
\lefteqn{\sum_{P\in\mathscr{M}_{n}^{m,k}}\prod_{i\geq1}\left\{\frac{\phi_i(1)}{i!}\right\}^{u_i(P)}
\prod_{i\geq1}\left\{\frac{\psi_{i-1}(1)}{(i-1)!}\right\}^{h_i(P)} }\\
&=&\sum_{j=0}^k\sum_{\ell=j}^{k}(-1)^{\ell-j}\binom{\ell-1}{\ell-j}\frac{\psi_{k-\ell}(\ell)}{(k-\ell)!}
{\binom{m+j}{j}}\frac{\phi_m(m+j+1)}{(m+1)!}.
\end{eqnarray*}
\end{corollary}
\begin{example}
Corollary \ref{coro 2.11}, in the case
$\psi_n(x)=x(x+1)\cdots(x+n-1)$, by (\ref{eqn 2.8}), produces
\begin{eqnarray}\label{eqn 2.10}
\sum_{P\in\mathscr{M}_{n}^{m,k}}\prod_{i\geq1}\left\{\frac{\phi_i(1)}{i!}\right\}^{u_i(P)}
={\binom{m+k}{k}}\frac{\phi_m(m+k+1)}{(m+1)!}.
\end{eqnarray}
In addition, let $\phi_n(x)=x(x-qn)^{n-1}$ for fixed $q$. Then
(\ref{eqn 2.10}) generates
\begin{eqnarray*}
\sum_{P\in\mathscr{M}_{n}^{m,k}}\prod_{i\geq1}\left\{\frac{(1-qi)^{i-1}}{i!}\right\}^{u_i(P)}
&=&{\binom{m+k+1}{k}}\frac{((1-q)m+k+1)^{m-1}}{m!},
\end{eqnarray*}
which, in the case $q=-(r+1)$, coincides with Example \ref{ex 2.10}.

Let $\phi_n(x)=\sum_{i=0}^nS(n,i)x^i$, which implies $\phi_n(1)$ is
the $n$-th Bell number $B_n$. Then (\ref{eqn 2.10}) produces
\begin{eqnarray*}
\sum_{P\in\mathscr{M}_{n}^{m,k}}\prod_{i\geq1}\left\{\frac{B_i}{i!}\right\}^{u_i(P)}
&=&{\binom{m+k}{k}}\frac{\sum_{i=0}^mS(m,i)(m+k+1)^i}{(m+1)!}.
\end{eqnarray*}
\end{example}

\section{Special Motzkin paths: Compositions}
A {\em composition} of nonnegative integer $\lambda$ into $j$ {\em
parts} is an ordered sequence $\lambda_1,\lambda_2,\dots,\lambda_j$
of length $j$ such that
$\lambda=\lambda_1+\lambda_2+\cdots+\lambda_j$ with each
$\lambda_i\geq 0$. Each $\lambda_i$ is called the $i$-th {\em
summand} of the composition. Compositions are well known
combinatorial objects \cite{andrews, carlitz, comtet} and several of
their properties have been discussed in some recent papers, as in
\cite{bjorner,heubach,heubman,knopf,knopfrob,merlunci}.

A composition can be regarded as a special Motzkin path if each
summand $\lambda_i$ is replaced by $u^{\lambda_i}d^{\lambda_i}$ when
$\lambda_i\geq 1$ and by a $h$ when $\lambda_i=0$. A $u$-segment or
$h$-segment of a composition is defined to be that of its
corresponding Motzkin path.

Let $\mathscr{C}_{m,k,j}$ denote the set of compositions of $m$ with
$j$ parts and $k$ zero summands, so any $C\in \mathscr{C}_{m,k,j}$
has $j-k$ $u$-segments. Define the ordinary generating functions for
weighted compositions according to the statistics $u_1,u_2,\ldots$
and $h_1,h_2,\ldots$ as follows
\begin{eqnarray*}
C_j(x,y;{\bf t;s})&=&\sum_{m,k\geq0}x^my^k
\sum_{C\in\mathscr{C}_{m,k,j}}\prod_{i\geq1}t_i^{u_i(C)}\prod_{i\geq1}s_i^{h_i(C)}\\
C(x,y;{\bf t;s;q})&=&\sum_{j\geq0}q^jC_j(x,y;{\bf t;s}).
\end{eqnarray*}

\begin{proposition}\label{pro3.1}
The explicit formula for $C(x,y;{\bf t;s;q})$ is
\begin{eqnarray}\label{eqn 3.1}
C(x,y;{\bf t;s;q})&=&\frac{S(qy)}{1+qS(qy)-qS(qy)T(x)},
\end{eqnarray}
where $T(x)=1+\sum_{i\geq 1}t_ix^i$, $S(y)=1+\sum_{i\geq 1}s_iy^i$.
\end{proposition}
\pf A recurrence relation for $C_j(x,y;{\bf t;s})$ can be derived as
follows
\begin{eqnarray*}
C_j(x,y;{\bf
t;s})=s_jy^j+\sum_{i=1}^{j}s_{j-i}y^{j-i}C_{i-1}(x,y;{\bf
t;s})(T(x)-1),
\end{eqnarray*}
for $j\geq 1$ and $C_0(x,y;{\bf t;s})=1$ if one notices that a
composition begins with a $h$-segment of length $i$ for $0\leq i\leq
j$ or a $u$-segment of length $r$ for $r\geq 1$. Then
\begin{eqnarray*}
C(x,y;{\bf t;s;q})&=&\sum_{j\geq0}q^jC_j(x,y;{\bf t;s})\\
&=&1+\sum_{j\geq1}q^j\left\{s_jy^j+\sum_{i=1}^{j}s_{j-i}y^{j-i}C_{i-1}(x,y;{\bf
t;s})(T(x)-1)\right\}\\
&=&S(qy)(1+q(T(x)-1)C(x,y;{\bf t;s;q})),
\end{eqnarray*}
which leads to (\ref{eqn 3.1}). \qed

\begin{theorem}\label{theo 3.2}
For any integers $m,k,j\geq 0$, there holds
\begin{eqnarray*}
\sum_{C\in\mathscr{C}_{m,k,j}}\prod_{i\geq1}t_i^{u_i(C)}\prod_{i\geq1}s_i^{h_i(C)}
&=&\frac{{\bf
P}_{k}^{(j-k+1)}\big(1!s_1,2!s_2,\cdots\big)}{k!}\frac{(j-k)!{\bf
B}_{m,j-k}\big(1!t_1,2!t_2,\cdots\big)}{m!}.
\end{eqnarray*}
\end{theorem}
\begin{proof}
By the definition of potential polynomials and (\ref{eqn 3.1}),
using the identity
\begin{eqnarray*}
\sum_{i=0}^{j-k}(-1)^i\binom{j-k}{i}\binom{j-k-i}{r}=\delta_{r,j-k},
\end{eqnarray*}
where $\delta_{r,j-k}$ is the Kronecker symbol, we have
\begin{eqnarray*}
\lefteqn{\sum_{C\in\mathscr{C}_{m,k,j}}\prod_{i\geq1}t_i^{u_i(C)}\prod_{i\geq1}s_i^{h_i(C)}=
[x^{m}y^kq^j]C(x,y;{\bf t;s;q})}\\
&=&[x^{m}y^kq^j]\frac{S(qy)}{1+qS(qy)-qS(qy)T(x)}=[x^{m}y^kq^j]\sum_{i\geq 0}q^iS(qy)^{i+1}(T(x)-1)^i\\
&=& [x^{m}(qy)^{k}]S(qy)^{j-k+1}(T(x)-1)^{j-k}=[(qy)^{k}]S(qy)^{j-k+1}[x^{m}](T(x)-1)^{j-k}\\
&=&[(qy)^{k}]S(qy)^{j-k+1}\sum_{i=0}^{j-k}(-1)^{i}\binom{j-k}{i}[x^{m}]T(x)^{j-k-i}\\
&=&\frac{{\bf
P}_{k}^{(j-k+1)}\big(1!s_1,2!s_2,\cdots\big)}{k!}\sum_{i=0}^{j-k}(-1)^i\binom{j-k}{i}\frac{{\bf
P}_{m}^{(j-k-i)}\big(1!t_1,2!t_2,\cdots\big)}{m!}\\
&=&\frac{{\bf
P}_{k}^{(j-k+1)}\big(1!s_1,2!s_2,\cdots\big)}{k!}\sum_{i=0}^{j-k}(-1)^i\binom{j-k}{i}\sum_{r=0}^m\frac{r!}{m!}
\binom{j-k-i}{r}{\bf B}_{m,r}\big(1!t_1,2!t_2,\cdots\big)\\
&=&\frac{{\bf
P}_{k}^{(j-k+1)}\big(1!s_1,2!s_2,\cdots\big)}{k!}\sum_{r=0}^m\frac{r!}{m!}
{\bf B}_{m,r}\big(1!t_1,2!t_2,\cdots\big)\sum_{i=0}^{j-k}(-1)^i\binom{j-k}{i}\binom{j-k-i}{r}\\
&=&\frac{{\bf
P}_{k}^{(j-k+1)}\big(1!s_1,2!s_2,\cdots\big)}{k!}\frac{(j-k)!{\bf
B}_{m,j-k}\big(1!t_1,2!t_2,\cdots\big)}{m!},
\end{eqnarray*}
as claimed.
\end{proof}

{\em Remark: Theorem \ref{theo 3.2} provides a unified method to
investigate compositions. This very general framework can be applied
to many special cases by specializing the parameters. For examples,
let $T(x)=\frac{1}{1-x}-x^r$, i.e., $t_i=1$ except for $t_r=0$ for
$i\geq 1$, then Theorem \ref{theo 3.2} in the case $k=0$ leads to
compositions without occurrences of $r$ \cite{chinn}; More
generally, let $T(x)=1+\sum_{i\in A}x^i$, where $A$ is a given set
of positive integers, then Theorem \ref{theo 3.2} in the case $k=0$
leads to compositions with summands in a given set \cite{heubach}.}

Recall that any $C\in \mathscr{C}_{m,k,j}$ has $j-k$ $u$-segments.
Let $\mathscr{C}_{m,k,j}^{\ell}$ be the subset of
$\mathscr{C}_{m,k,j}$ with $\ell$ number of $h$-segments. Note that
${\bf B}_{m,r}\big(1!qt_1,2!qt_2,\cdots\big)=q^r{\bf
B}_{m,r}\big(1!t_1,2!t_2,\cdots\big)$ by (\ref{eqn 1.5}), combining
(\ref{eqn 1.3}) with Theorem \ref{theo 3.2}, we have
\begin{corollary}\label{coro 3.3}
For any integers $m,k,\ell\geq 0$, there holds
\begin{eqnarray*}
\sum_{C\in\mathscr{C}_{m,k,j}^{\ell}}\prod_{i\geq1}t_i^{u_i(C)}\prod_{i\geq1}s_i^{h_i(C)}
=\binom{j-k+1}{\ell}\frac{(j-k)!\ell!}{k!m!}{\bf
B}_{m,j-k}\big(1!t_1,2!t_2,\cdots\big){\bf
B}_{k,\ell}\big(1!s_1,2!s_2,\cdots\big).
\end{eqnarray*}
\end{corollary}
Using (\ref{eqn 1.5}) and then comparing the coefficient of
$t_1^{r_1}t_2^{r_2}\cdots t_m^{r_m}s_1^{\ell_1}s_2^{\ell_2}\cdots
s_k^{\ell_k}$ in Corollary \ref{coro 3.3}, one can obtain that
\begin{corollary}
The number of compositions in $\mathscr{C}_{m,k,j}^{\ell}$ with
$u$-segments of type $1^{r_1}2^{r_2}\cdots m^{r_m}$ and $h$-segments
of type $1^{\ell_1}2^{\ell_2}\cdots k^{\ell_k}$ is
\begin{eqnarray*}
\binom{j-k+1}{\ell}\binom{j-k}{r_1,r_2,\cdots,r_m}\binom{\ell}{\ell_1,\ell_2,\cdots,\ell_k}.
\end{eqnarray*}
\end{corollary}

\vskip0.5cm
\section{Matrix Compositions}

An {\em matrix composition} of nonnegative integer $m$ is a $p\times
j$ matrix $M$ with nonnegative entries such that $m$ is the sum of
entries of $M$ for some $p,j\geq 0$. For any $p\times j$ matrix
composition $M$, each row of $M$ can be regarded as a special
Motzkin path, just as the case $p=1$ considered in Section 3.

Recall that $C_j(x,y;{\bf t;s})$ is the ordinary generating
functions for weighted compositions with $j$ parts according to the
statistics $u_1,u_2,\ldots$ and $h_1,h_2,\ldots$. Then the ordinary
generating functions for weighted $p\times j$ matrix compositions
according to the statistics $u_1,u_2,\ldots$ and $h_1,h_2,\ldots$ is
just $C_j^p(x,y;{\bf t;s})$. From Proposition \ref{pro3.1}, one can
deduce easily that
\begin{eqnarray*}
C_j(x,y;{\bf t;s})=\sum_{i=0}^j\frac{y^{j-i}{\bf
P}_{j-i}^{(i+1)}\big(1!s_1,2!s_2,\cdots\big)}{(j-i)!}(T(x)-1)^i.
\end{eqnarray*}
However, it seems that the coefficients of $[x^my^k]$ in
$C_j^p(x,y;{\bf t;s})$ have no simple explicit formulas. For the
sake of this, we can consider a kind of special matrix compositions,
called {\em bipartite matrix compositions}, namely, each row has the
type $(a_1,a_2,\dots,a_i, 0,\dots,0)$ for some $0\leq i\leq j$,
where $a_1,\dots,a_i\geq 1$. If $a_1=\cdots=a_i=1$, then we call it
a {\em bipartite $(0,1)$-matrix}. Let $\mathscr{B}_{m,p,j}$ denote
the set of $p\times j$ bipartite matrix compositions of $m$ and let
$B_{p,j}(x;{\bf t})$ denote the ordinary generating functions for
weighted $p\times j$ bipartite matrix compositions according to the
statistics $u_1,u_2,\ldots$, that is,
\begin{eqnarray*}
B_{p,j}(x;{\bf t})=\sum_{m\geq
0}x^m\sum_{B\in\mathscr{B}_{m,p,j}}\prod_{i\geq1}t_i^{u_i(B)}.
\end{eqnarray*}

\begin{proposition}\label{pro4.1}
The explicit formula for $B_{p,j}(x;{\bf t})$ is
\begin{eqnarray}\label{eqn 4.1}
B_{p,j}(x;{\bf
t})&=&\left\{\frac{1-\{(T(x)-1)\}^{j+1}}{1-\{(T(x)-1)\}}\right\}^p,
\end{eqnarray}
where $T(x)=1+\sum_{i\geq 1}t_ix^i$.
\end{proposition}
\begin{proof}
For any $1\times j$ bipartite matrix compositions, it has the type
$(a_1,a_2,\dots,a_i, 0,\dots,0)$ for some $0\leq i\leq j$, where
$a_1,a_2,\dots,a_i\geq 1$, then each $a_r$ has the weight $t_{a_r}$
which is a term of $T(x)-1$. Hence we have
\begin{eqnarray*}
B_{1,j}(x;{\bf
t})=\sum_{i=0}^j(T(x)-1)^i=\frac{1-\{(T(x)-1)\}^{j+1}}{1-\{(T(x)-1)\}}.
\end{eqnarray*}
Then by the relation $B_{p,j}(x;{\bf t})=B_{1,j}(x;{\bf t})^p$, we
obtain (\ref{eqn 4.1}).
\end{proof}

\begin{theorem}\label{theo 4.2}
For any integers $m,p,j\geq 0$, there holds
\begin{eqnarray*}
\sum_{B\in\mathscr{B}_{m,p,j}}\prod_{i\geq1}t_i^{u_i(B)}=\sum_{r=0}^m\frac{r!}{m!}U_{p,j,r}{\bf
B}_{m,r}\big(1!t_1,2!t_2,\cdots\big),
\end{eqnarray*}
where
$$U_{p,j,r}=\sum_{i=0}^{[\frac{r}{j+1}]}(-1)^i\binom{p}{i}\binom{p+r-i(j+1)-1}{p-1}.$$
\end{theorem}
\begin{proof}
Similar to the proof of Theorem \ref{theo 3.2}, we have
\begin{eqnarray*}
\lefteqn{\sum_{B\in\mathscr{B}_{m,p,j}}\prod_{i\geq1}t_i^{u_i(B)}=[x^{m}]B_{p,j}(x;{\bf
t})=[x^m]\left\{\frac{1-\{(T(x)-1)\}^{j+1}}{1-\{(T(x)-1)\}}\right\}^p}\\
&=&\sum_{r=0}^m\sum_{i=0}^{[\frac{r}{j+1}]}(-1)^i\binom{p}{i}\binom{p+r-i(j+1)-1}{p-1}[x^{m}](T(x)-1)^{r}\\
&=&\sum_{r=0}^mU_{p,j,r}\frac{r!{\bf
B}_{m,r}\big(1!t_1,2!t_2,\cdots\big)}{m!},
\end{eqnarray*}
as claimed.
\end{proof}

Let $\mathscr{B}_{m,p}^{j,r}$ be the subset of $\mathscr{B}_{m,p,j}$
with $r$ number of nonzero entries. Note that ${\bf
B}_{m,r}\big(1!qt_1,2!qt_2,\cdots\big)=q^r{\bf
B}_{m,r}\big(1!t_1,2!t_2,\cdots\big)$ by (\ref{eqn 1.5}), combining
(\ref{eqn 1.3}) with Theorem \ref{theo 4.2}, we have
\begin{corollary}\label{coro 4.3}
For any integers $m,p,j,r\geq 0$, there holds
\begin{eqnarray*}
\sum_{B\in\mathscr{B}_{m,p}^{j,r}}\prod_{i\geq1}t_i^{u_i(B)}
=\frac{r!{\bf B}_{m,r}\big(1!t_1,2!t_2,\cdots\big)U_{p,j,r}}{m!}.
\end{eqnarray*}
\end{corollary}
Using (\ref{eqn 1.5}) and comparing the coefficient of
$t_1^{r_1}t_2^{r_2}\cdots t_m^{r_m}$ in Corollary \ref{coro 4.3},
one can obtain that
\begin{corollary}
The number of $p\times j$ bipartite matrix compositions of $m$ in
$\mathscr{B}_{m,p}^{j,r}$ with nonzero entries of type
$1^{r_1}2^{r_2}\cdots m^{r_m}$ is
\begin{eqnarray*}
\binom{r}{r_1,r_2,\cdots,r_m}\sum_{i=0}^{[\frac{r}{j+1}]}(-1)^i\binom{p}{i}\binom{p+r-i(j+1)-1}{p-1}.
\end{eqnarray*}
\end{corollary}

\begin{example}
Let $T(x)=1+x$, then Theorem \ref{theo 4.2} signifies that the
number of $p\times j$ bipartite matrix compositions of $m$ with
nonzero summands $1$ or, in other words, of $p\times j$ bipartite
(0,1)-matrices with $m$ ones is counted by $U_{p,j,m}$. Specializing
to $p=m+1$, we have
$$U_{m+1,j,m}=\sum_{i=0}^{[\frac{m}{j+1}]}(-1)^i\binom{m+1}{i}\binom{2m-i(j+1)}{m}.$$
Note that the number $\frac{1}{m+1}U_{m+1,j,m}$ counts the unlabeled
plane trees on $m+1$ vertices in which every vertex has outdegree
not greater than $j$. Klarner \cite{klarner} first considered this
problem, which was solved by Chen \cite{chen90} and later by Mansour
and Sun \cite{mansun1}. Then it is clear that $U_{m+1,j,m}$ counts
the unlabeled double rooted plane trees on $m+1$ vertices in which
every vertex has outdegree not greater than $j$. We leave it as an
open problem to find the bijection between these two settings.
\end{example}

\section*{Acknowledgements} The authors are grateful to the anonymous referees for the
helpful suggestions and comments. The work is supported by The
National Science Foundation of China (10726021).



\vskip1cm

\end{document}